\documentclass[10pt,oneside,leqno]{amsart}
\usepackage{amsxtra}
\usepackage{amsopn}
\usepackage{color}
\usepackage{amsmath,amsthm,amssymb}
\usepackage{amscd}
\usepackage{amsfonts}
\usepackage{latexsym}
\usepackage{verbatim}

\theoremstyle{plain}
\newtheorem{theorem}{Theorem}[section]
\newtheorem{definition}[theorem]{Definition}
\newtheorem{lemma}[theorem]{Lemma}
\newtheorem{proposition}[theorem]{Proposition}
\newtheorem{corollary}[theorem]{Corollary}
\newtheorem{remark}[theorem]{Remark}
\newtheorem{example}[theorem]{Example}
\newtheorem{question}[theorem]{QUESTION}
\newtheorem{remark-question}[section]{Remark-Question}
\newtheorem{conjecture}[section]{Conjecture}

\newcommand\C{{\mathbb C}}

\newcommand\CP{{\mathbb {CP}}}

\newcommand\R{{\mathbb R}}
\newcommand\Z{{\mathbb Z}}

\newcommand\T{{\mathbb T}} \newcommand\Proj{{\mathbb P}}

\newcommand\frg{{\mathfrak g}}
\newcommand\frh{{\mathfrak h}}

\newcommand\Real{{\mathfrak R}{\frak e}\,} 

\newcommand\db{{\bar{\partial}}}
\newcommand\zzz{{\!\!\!}}

\begin{document}
\title[]{On generalized Gauduchon metrics}

\subjclass[2000]{
}

\author{Anna Fino}
\address[Fino]{Dipartimento di Matematica\\
 Universit\`a di Torino\\
Via Carlo Alberto 10\\
10123 Torino, Italy} \email{annamaria.fino@unito.it}

\author{Luis Ugarte}
\address[Ugarte]{Departamento de Matem\'aticas\,-\,I.U.M.A.\\
Universidad de Zaragoza\\
Campus Plaza San Francisco\\
50009 Zaragoza, Spain} \email{ugarte@unizar.es}

\maketitle

\begin{abstract} We study  a class of Hermitian metrics on complex manifolds, recently introduced by J. Fu, Z. Wang
 and  D. Wu,  which are a generalization of Gauduchon metrics. This class includes the one of Hermitian metrics for which the associated fundamental $2$-form is $\partial \overline \partial$-closed. Examples are  given on nilmanifolds, on products of Sasakian manifolds, on $S^1$-bundles and via the twist construction introduced by A. Swann.
\end{abstract}


\section{Introduction}

Let $(M,J,g)$ be a Hermitian manifold of real  dimension $2n$.
If the fundamental $2$-form $\Omega (\cdot, \cdot) = g(J \cdot, \cdot)$ is $d$-closed, then the metric
$g$ is K\"ahler. In the literature, weaker conditions on $\Omega$ have been studied and they involve
the closure with respect to the $\partial \overline \partial$-operator of the
$(k , k)$-form $\Omega^k = \Omega \wedge \cdots \wedge\Omega$.

If $\partial \overline \partial \Omega =0$, then the
Hermitian structure $(J, g)$ is said to be {\em strong K\"ahler with torsion} and
$g$ is called {\em SKT}  (or  also {\em pluriclosed} )(see e.g. \cite{GHR}). In this case the Hermitian structure is characterized
by the condition that torsion $3$-form $c = J d \Omega$ of the   Bismut connection is $d$-closed.
SKT metrics have been recently studied by many
authors and they have also applications in type II string theory
and in 2-dimensional supersymmetric $\sigma$-models
\cite{GHR,Str, IP}. Moreover, they have also links with generalized
K\"ahler structures (see for instance
\cite{GHR,Gu,Hi2,AG}). New simply-connected SKT examples have been
constructed by A. Swann in \cite{Sw} via the twist construction, by reproducing the $6$-dimensional examples
found previously in \cite{GGP}. Recently, Streets and Tian introduced a Hermitian Ricci flow under which the SKT condition is preserved \cite{ST}.

If $\partial \overline \partial \Omega^{n - 2} =0$, then the Hermitian
metric $g$ on $M$ is said to be {\em astheno-K\"ahler}. Jost and Yau used this condition in \cite{JY} to study Hermitian harmonic maps and to extend Siu's rigidity theorem to non-K\"ahler manifolds.

On a complex surface any Hermitian metric is automatically astheno-K\"ahler and in
complex dimension $n =3$ the notion of astheno-K\"ahler metric
coincides with that one of SKT. Astheno-K\"ahler structures on  products of Sasakian manifolds and then in particular on Calabi-Eckmann manifolds have been constructed in \cite{M}.
For $n > 3$ other examples of astheno-K\"ahler manifolds have been found  in \cite{FT} via the twist construction  \cite{Sw}, blow-ups and resolutions of orbifolds.

In \cite{FT2} it was shown that the blow-up of an  SKT  manifold $M$ at
a point or along a compact complex submanifold $Y$ is still SKT, as in the K\"ahler case
(see for example \cite{Bl}). Moreover,  by \cite{FT}  the same property holds for the blow-ups of complex manifolds
endowed with a $J$-Hermitian metric such that
$\partial\overline{\partial}\,\Omega =0$ and $\partial\overline{\partial}\,\Omega^2=0$.

If $\Omega^{n-1}$ is
$\partial \overline\partial$-closed or equivalently if the Lee form is co-closed, then the Hermitian metric $g$ is
called {\emph {standard}} or a {\em Gauduchon metric} \cite{Gau}.
Recently, in \cite{FWW} J. Fu,  Z. Wang and D. Wu introduced a generalization of the Gauduchon metrics on complex manifolds. Let  $(M, J)$ be a   complex manifold $M$ of complex dimension $n$ and let $k$  be an integer such that $1 \leq k \leq n-1$,  a $J$-Hermitian metric $g$ on $M$  is called {\em $k$-th Gauduchon} if
\begin{equation} \label{defk-Gaud}
\partial \overline \partial \Omega^k  \wedge \Omega^{n -k-1} =0.
\end{equation}
Then, by definition  the notion of  $(n-1)$-th Gauduchon metric coincides with the one of the usual Gauduchon metric. In \cite{FWW} it is  associated   to any Hermitian structure  $\Omega$  on a complex manifold $(M, J)$ a unique constant $\gamma_k (\Omega)$, which is invariant by biholomorphisms and which depends smoothly on $\Omega$. It is also proved that $\gamma_k (\Omega) =0$ if and only if there exists a $k$-th Gauduchon metric in the conformal class of $\Omega$. Moreover, the first examples of $1$-st Gauduchon metrics have been constructed in \cite{FWW} on the three-dimensional complex manifolds constructed by Calabi and on $S^5\times S^1$.

In this paper we will study $1$-st Gauduchon metrics, i.e.  Hermitian metrics for which $$
\partial \overline \partial \Omega  \wedge \Omega^{n -2} =0.
$$
This class of  metrics includes the SKT ones as particular case, so it is natural to study  which properties that hold for the SKT metrics are still valid in the case of $1$-st Gauduchon ones.

Astheno-K\"ahler and SKT metrics on compact complex manifolds cannot be balanced for $n > 2$ unless they are
K\"ahler (see \cite{MT,AI}), where by balanced one means that
the Lee form vanishes.  If the Lee form is exact, then the  Hermitian structure is called conformally balanced.
By \cite{IP,P, FT} a conformally balanced SKT (or astheno-K\"ahler) structure
on a compact manifod of complex dimension $n$ whose Bismut connection has (restricted)
holonomy contained in $SU(n)$ is necessarily K\" ahler.
In Section~\ref{sectionbalanced} we prove similar results for $1$-st Gauduchon metrics.

In complex dimension $3$ invariant SKT structures on {\em nilmanifolds}, i.e. on compact quotients
of nilpotent Lie groups by uniform discrete subgroups, were studied
in \cite{FPS,U} showing that the existence of such a
structure depends only on the left-invariant complex structure on the Lie group and that the Lie group is $2$-step.  Nilmanifolds of any even dimension  endowed with a left-invariant complex structure $J$ and an SKT $J$-Hermitian metric were studied in \cite{EFV} showing that they are all $2$-step and they are indeed the total space of holomorphic principal torus  bundles over complex tori.

In Section \ref{nilmanifolds} we  prove that on a $6$-dimensional nilmanifold endowed with an invariant complex structure $J$, an invariant $J$-Hermitian metric is $1$-st Gauduchon if and only if it is SKT.  However, by using the results in \cite{FWW} we show that there are complex $6$-dimensional nilmanifolds admitting non-invariant $1$-st Gauduchon metrics. A family   $J_t$, $t \in (0,1]$,  of  complex structures admitting compatible  $1$-st Gauduchon metrics for any $t$    and such  that $J_t$ admits SKT metrics only for $t =1$  is also given on a $6$-dimensional nilmanifold. Moreover, we show that this family  of complex structures on the previous nilmanifold   does not have any compatible  (invariant or not) balanced metric.

 The situation is different in higher dimension since  in Section \ref{products} we give examples in dimension 8 of complex nilmanifolds admitting invariant $1$-st Gauduchon metrics and no SKT metrics. These examples are products of (quasi)-Sasakian manifolds. We construct $1$-st Gauduchon metrics on $S^1$-bundles  over  quasi-Sasakian manifolds and on the product of two Sasakian manifolds endowed with  the complex structure introduced in \cite{M}.
In Section \ref{blow-ups}  we  study the existence of $1$-st Gauduchon metrics  on the  blow-up of a complex manifold at a point or along a compact submanifold. In the last section,  by using the twist construction \cite{Sw}, we find  examples of  simply-connected $6$-dimensional compact complex manifolds endowed with $1$-st Gauduchon metrics.

\section{$1$-st Gauduchon metrics and relation with the balanced condition} \label{sectionbalanced}

In this section we show that $1$-st Gauduchon metrics, as the SKT ones, are complementary to the balanced condition.
Let  us  start to review the definition and the properties of the $k$-th Gauduchon metrics contained  in \cite{FWW}.

We recall the following
\begin{definition} \cite{FWW}  Let  $(M, J)$ be a complex manifold  of complex dimension $n$ and let $k$ be an integer such that $1 \leq k \leq n -1$. A $J$-Hermitian metric  $g$   on $(M, J)$  is called {\em $k$-th Gauduchon}  if  its fundamental $2$-form $\Omega$ satisfies the condition \eqref{defk-Gaud}.
\end{definition}

For $k = n-1$ one gets the classical standard  metric. Moreover, according to  the previous definition, an SKT metric is a $1$-st Gauduchon metric and an astheno-K\"ahler metric is a $(n-2)$-th Gauduchon metric.

Extending the result proved by Gauduchon in \cite{Gau} for standard metrics  it is shown in \cite{FWW} that if $(M,J, g, \Omega)$ is a $n$-dimensional compact Hermitian manifold, then for any integer $1 \leq k \leq n - 1$, there exists a unique constant $\gamma_k (\Omega)$ and a   (unique up to a constant) function $v \in {\mathcal C}^{\infty} (M)$  such that
$$
\frac{i}{2} \partial \overline \partial (e^v \Omega^k) \wedge \Omega^{n - k -1} = \gamma_k (\Omega) e^v \Omega^n.
$$
If $(M,J, g, \Omega)$ is K\"ahler, then $\gamma_k(\Omega) =0$ and $v$ is a constant function for any $1 \leq k \leq n - 1$.

The constant $\gamma_k(\Omega)$ is invariant under biholomorphisms and by \cite[Proposition 11]{FWW} the sign of  $\gamma_k (\Omega)$ is invariant in the conformal class of $\Omega$.

To compute the sign of 
the constant $\gamma_k(\Omega)$ one can use the following
\begin{proposition} \cite{FWW} For a Hermitian structure $(J, g, \Omega)$ on an $n$-dimensional complex manifold $(M, J)$, the number $\gamma_k (\Omega)$ is $>0$ ($= 0$, or $< 0$) if and only if there exists a metric $\tilde g$  in the conformal class of $g$ such that
$$
\frac{i}{2} \partial \overline \partial  \tilde \Omega^k \wedge \tilde \Omega^{n - k - 1} > 0  \, (=0, \, {\mbox {or}} \, < 0),
$$
where $\tilde  \Omega$ is the   fundamental $2$-form associated to $(J, \tilde g)$.
\end{proposition}

For $n = 3$, by \cite[Theorem 6]{FWW} one has the following

\begin{proposition}\label{curve} \cite{FWW} On any  compact $3$-dimensional complex manifold there exists a Hermitian metric $g$ such that  its fundamental $2$-form $\Omega$ has $\gamma_1 (\Omega) >0$.
\end{proposition}

 So for $n =3$ if one finds a Hermitian metric $\tilde g$ such that its fundamental $2$-form $ \tilde \Omega$ has $\gamma_1 (\tilde \Omega) <0$, by using \cite[Corollary 10]{FWW} then there exists a $1$-st Gauduchon metric.

\smallskip

We recall that the Lee form of a Hermitian manifold $(M,J,g)$ of complex dimension $n$ is the
1-form
$$\theta =J*d*\Omega=Jd^*\Omega,$$
where $d^*$ is the formal adjoint of $d$ with
respect to $g$.
The formula $*\Omega=\Omega^{n-1}/(n-1)!$ implies that $d(\Omega^{n-1})=\theta \wedge \Omega^{n-1}$.
The Hermitian structure  $(J, g)$  is called \emph{balanced} if $\theta$ vanishes or equivalently if $d(\Omega^{n-1})=0$ and such structures belong that  the class $W_3$ in the well known Gray-Hervella classification \cite{GH}.

In real dimension $4$, the SKT condition is equivalent to $d^*\theta=0$, but in higher dimension
both  the SKT  and  the astheno-K\"ahler condition on a compact Hermitian manifold are complementary to the
balanced one \cite{AI,MT}. In the case of $1$-st Gauduchon metrics we can prove  the following

\begin{proposition} Let $(M, J)$ be  a compact complex manifold of complex dimension $n \geq 3$. If $g$ is  a $J$-Hermitian metric which is $1$-st Gauduchon and balanced, then $g$ is K\"ahler.
\end{proposition}

\begin{proof}
For  a  compact Hermitian manifold $(M, J, g)$ one has the two
natural linear operators acting on  differential forms:
$$
L \varphi = \Omega \wedge \varphi,
$$
and the adjoint operator $L^*$ of $L$ with respect to the global scalar product defined by
$$
< \varphi, \psi > =p! \int_M (\varphi, \psi) {\mbox {vol}}_g,
$$
where $\Omega$ is the fundamental $2$-form of $(J, g)$, $(\varphi, \psi)$ is the pointwise $g$-scalar product and ${\mbox {vol}}_g$ is the volume form.
By \cite{MT}  one has
\begin{equation}\label{leeform}
 {L^*}^{3} (2 i \partial \overline \partial \Omega  \wedge \Omega)= 96 (n -2)
[2 d^*  \theta + 2 \vert \vert \theta \vert \vert^2 - \vert \vert T \vert \vert^2 ],
\end{equation}
where $\theta$ is the Lee form, $d^* \theta$ its
co-differential, $\vert \vert \theta \vert \vert$ its $g$-norm and $T$ is the torsion of the Chern
connection $\nabla^C$. \newline
By using the formula
$$
{L^*}^r L^s \varphi = L^s  {L^*}^r \varphi + \sum_{i = 1}^r 4^i (i!)^2 \left ( \begin{array}{c} s\\ i \end{array}  \right ) \left ( \begin{array}{c} r\\ i \end{array}  \right ) \left ( \begin{array}{c} n - p - s + r\\ i \end{array}  \right ) L^{s -i} {L^*}^{r - i} \varphi,
$$
which holds for any $p$-form $\varphi$  and any  positive integer   $r \leq s$, one gets
\begin{equation}\label{leeformappll}
{L^{*}}^n (2 i \partial \overline \partial \Omega \wedge \Omega^{n -2}) =
{L^{*}}^n (L^{n - 3} (2 i \partial \overline \partial \Omega \wedge \Omega)) =
4^{n} \frac{n !}{3 !} (n - 3) ! {L^*}^3  (2 i \partial \overline \partial \Omega  \wedge \Omega).
\end{equation}
Therefore, if $(J, g)$ is balanced, then $\theta =0$ and, consequently,
the $1$-st Gauduchon
condition implies that $T =0$, i.e. $g$ is K\"ahler.
\end{proof}

A Hermitian structure is called \emph{conformally balanced} if the Lee form $\theta$ is $d$-exact.
By \cite{IP,P,FT}, a conformally balanced SKT (or astheno-K\"ahler) structure on a compact manifold of complex dimension $n$ whose Bismut
connection has (restricted) holonomy contained in $SU(n)$ is necessarily K\"ahler. We  can now prove a similar result for
the $1$-st Gauduchon  metrics.

\begin{theorem}\label{conform-balanced}
A conformally balanced $1$-st Gauduchon structure $(J, g)$ on a compact manifold of
complex dimension $n \geq 3$ whose Bismut connection has (restricted) holonomy contained in $SU(n)$ is
necessarily K\"ahler
and therefore it is a Calabi-Yau structure.
\end{theorem}
\begin{proof} Since the Hermitian structure is $1$-st Gauduchon, then by \eqref{leeform} and \eqref{leeformappll}, we have
$$
2  d^* \theta + 2 \vert \vert \theta \vert \vert^2 - \vert \vert T \vert \vert^2 =0\,.
$$
Therefore,
\begin{equation} \label{deltatheta}
d^* \theta = \frac{1} {2} \vert \vert T \vert \vert ^2 -  \vert \vert \theta \vert \vert ^2.
\end{equation}
By \cite[formula (2.11)]{AI}, the trace $2u$ of the Ricci form of the Chern connection is related to
the trace $b$ of the Ricci form of the Bismut connection by the equation
\begin{equation} \label{tracericci}
2 u = b + 2 d^* \theta + 2 \vert \vert \theta \vert \vert^2\,.
\end{equation}
Since  the condition that the Bismut connection has (restricted) holonomy contained
in $SU(n)$ implies that the Ricci form of the Bismut connection vanishes,
by using \eqref{deltatheta} and \eqref{tracericci}, we obtain
$$
2 u = \vert \vert T \vert \vert^2.
$$
Now, if $u >0$ then it follows from \cite[Theorem 4.1 and formula (4.4)]{IP} that
all plurigenera of $(M,J)$ vanish.
But, since $(J, g)$ is conformally balanced, there exists a
nowhere vanishing holomorphic $(n,0)$-form by
\cite{Str,P}. Therefore, $u$ must vanish and $g$ is K\"ahler.
\end{proof}

\section{Generalized Gauduchon metrics on 6-nilmanifolds} \label{nilmanifolds}

In this section we study $1$-st Gauduchon metrics  on complex nilmanifolds  $(M= \Gamma \backslash G , J)$ of real dimension six endowed
with an invariant complex structure $J$, that is,  $G$ is a simply-connected nilpotent Lie
group and $\Gamma$  is  a lattice in $G$ of maximal rank, and $J$ arises from a left-invariant complex structure on the Lie group $G$.

In \cite{Salamon}  S. Salamon proved  that, up to isomorphism, there are exactly eighteen nilpotent
Lie algebras of real dimension 6 admitting complex structures. In \cite{U} it was shown that in fact  there are two special and disjoint types of complex equations for the  $6$-dimensional nilpotent Lie algebras, depending
on the \lq \lq nilpotency" of the complex structure. We recall that  a complex structure $J$ on a $2n$-dimensional nilpotent Lie algebra $\mathfrak g$ is called nilpotent if there is a
basis of $(1,0)$-forms  $\{ \omega_j \}_{j = 1}^n$  satisfying $d \omega^1 =0$ and
$$
d \omega^j \in \Lambda^2 < \omega^1, \dots, \omega^{j -1}, \overline \omega^1, \dots, \overline \omega^{j -1} >, \quad j =2, \dots, n.
$$

In dimension six one has the following

\begin{proposition}\label{J-red}\cite{U, UV}
Let $J$ be an invariant complex structure on a  nilmanifold $M$ of real dimension six.
\begin{enumerate}
\item[(i)] If $J$ is nilpotent then there is a global basis
$\{\omega^j\}_{j=1}^3$ of  invariant $(1,0)$-forms on $M$ satisfying
\begin{equation}\label{nilpotentJ}
\left\{
\begin{array}{lcl}
d\omega^1 \zzz & = &\zzz 0,\\
d\omega^2 \zzz & = &\zzz \epsilon\, \omega^{1\bar{1}} \, ,\\
d\omega^3 \zzz & = &\zzz \rho\, \omega^{12} + (1-\epsilon)A\,
\omega^{1\bar{1}} + B\, \omega^{1\bar{2}} + C\, \omega^{2\bar{1}} +
(1-\epsilon)D\, \omega^{2\bar{2}},
\end{array}
\right.
\end{equation}
where $A,B,C,D\in \mathbb{C}$, and $\epsilon,\rho \in \{0,1\}$.
\item[(ii)] If $J$ is non-nilpotent then there is a global basis
$\{\omega^j\}_{j=1}^3$ of  invariant $(1,0)$-forms on $M$ satisfying
\begin{equation}\label{non-nilp-reducido}
\begin{cases}
\begin{array}{lcl}
d\omega^1 \zzz & = &\zzz 0,\\
d\omega^2 \zzz & = &\zzz  \omega^{13} + \omega^{1\bar{3}} \, ,\\
d\omega^3 \zzz & = &\zzz  i\epsilon\, \omega^{1\bar{1}} \pm i\,
(\omega^{1\bar{2}} - \omega^{2\bar{1}}) ,
\end{array}
\end{cases}
\end{equation}
with $\epsilon=0,1$.
\end{enumerate}
\end{proposition}

The fundamental 2-form $\Omega$ of any invariant
$J$-Hermitian metric $g$ on $M$ is then given by
\begin{equation}\label{2forma}
\Omega= \sum_{j,k=1}^3 x_{j\bar{k}}\, \omega^{j\bar{k}},
\end{equation}
where $x_{j\bar{k}} \in \mathbb{C}$ and $\overline{x_{k\bar{j}}}=-x_{j\bar{k}}$. Notice that
the positive definiteness of the metric $g$ implies  in particular that $$-i\, x_{j\bar{j}} \in \mathbb{R}^+, \quad \quad
i \det (x_{j\bar{k}}) >0.
$$

The following result is proved by a direct calculation, so we omit
the proof.

\begin{lemma}\label{partialOmega}
Let $(J,g)$ be a Hermitian structure on a $6$-dimensional nilpotent Lie algebra
$\frg$, and $\Omega$ its fundamental form.
\begin{enumerate}
\item[{\rm (i)}] If $J$ is nilpotent, then in terms of the basis
$\{\omega^j\}_{j=1}^3$
satisfying~\eqref{nilpotentJ}, the $(2,2)$-form $\partial\db\Omega$ is given
by
$$
\partial\db\Omega = x_{3\bar{3}} \left(\rho+|B|^2+|C|^2 - 2(1-\epsilon)\,
\Real(A\bar{D}) \right) \omega^{1\bar{1}2\bar{2}}.
$$
This implies that
$$
\frac{i}{2}\partial\db\Omega \wedge \Omega = \frac{i}{2} \frac{x_{3\bar{3}}^2}{\det (x_{j\bar{k}})} \left(\rho+|B|^2+|C|^2 - 2(1-\epsilon)\,
\Real(A\bar{D}) \right)\, \Omega^3,
$$
and therefore the sign of $\gamma_1(\Omega)$ depends only on the complex structure.
\item[{\rm (ii)}] If $J$ is non-nilpotent, then in terms of the
basis $\{\omega^j\}_{j=1}^3$
satisfying~\eqref{non-nilp-reducido}, the form $\partial\db\Omega$
is given by
$$
\partial\db\Omega = 2 x_{3\bar{3}}\, \omega^{1\bar{1}2\bar{2}}
+2 x_{2\bar{2}}\, \omega^{1\bar{1}3\bar{3}}.
$$
This implies that
$$
\frac{i}{2}\partial\db\Omega \wedge \Omega = i \frac{x_{2\bar{2}}^2 + x_{3\bar{3}}^2}{\det (x_{j\bar{k}})}\, \Omega^3,
$$
and therefore $\gamma_1(\Omega)>0$ for any $\Omega$.
\end{enumerate}
\end{lemma}

As an immediate consequence we get

\begin{proposition}\label{6dim-nilm}
Let $(M,J)$ be a 
complex nilmanifold of real dimension six endowed  with  an invariant complex structure $J$.
An invariant $J$-Hermitian metric on $M$
is $1$-st Gauduchon if and only if it is SKT.
\end{proposition}

The classification of nilmanifolds of real dimension $6$ admitting invariant SKT metric is given in \cite{FPS}.
It is clear from Lemma~\ref{partialOmega} that  $J$ has to be nilpotent and $\epsilon$ must be zero in order to have an invariant Hermitian
structure  $\Omega$ such that $\gamma_1(\Omega)<0$.
By \cite{U} it follows that if $\epsilon=0$ and $J$  is not bi-invariant then there is a basis $\{\omega^j\}_{j=1}^3$
such that
\begin{equation}\label{epsilonzero-red-new}
d \omega^1=d\omega^2=0,\quad\  d\omega^3=\rho\, \omega^{12} +
\omega^{1\bar{1}} + B\,\omega^{1\bar{2}} + (x+iy)\,\omega^{2\bar{2}},
\end{equation}
where $B,x+iy\in \mathbb{C}$, and $\rho= 0,1$.
The underlying  real  Lie algebra  $\mathfrak g$ is isomorphic to one of the following Lie algebras
$$
\begin{array}{ll}
\frh_2 =  (0,0,0,0,12,34)&
\frh_3 =  (0,0,0,0,0,12+34)\\[4pt]
\frh_4 =(0,0,0,0,12,14+23) &
\frh_5 =(0,0,0,0,13 + 42, 14 + 23)\\[4 pt]
\frh_6 =(0,0,0,0,12,13)&
\frh_8 =(0,0,0,0,0,12),
\end{array}
$$
where for instance by $(0,0,0,0,0,12)$ we denote the Lie algebra with structure equations $d e^j =0, j = 1, \dots, 5$, $d e^6 = e^1 \wedge e^2$.
The classes of isomorphisms can be distinguished by the  following conditions:
\begin{enumerate}
\item[{\rm (a)}] If $|B|=\rho$, then the Lie algebra $\frg$ is
isomorphic to
\begin{enumerate}
\item[{\rm (a1)}] $\frh_2$, for $y\not=0$;
\item[{\rm (a2)}]
$\frh_3$, for $\rho=y=0$ and $x\not=0$;
\item[{\rm (a3)}] $\frh_4$,
for $\rho=1$, $y=0$ and $x\not=0$;
\item[{\rm (a4)}] $\frh_6$, for
$\rho=1$ and $x=y=0$;
\item[{\rm (a5)}] $\frh_8$, for $\rho=x=y=0$.
\end{enumerate}
\item[{\rm (b)}] If $|B|\not=\rho$, then the Lie algebra $\frg$
is isomorphic to
\begin{enumerate}
\item[{\rm (b1)}] $\frh_2$, for $4y^2 >
(\rho-|B|^2)(4x+\rho-|B|^2)$;
\item[{\rm (b2)}] $\frh_4$, for
$4y^2 = (\rho-|B|^2)(4x+\rho-|B|^2)$;
\item[{\rm (b3)}] $\frh_5$,
for $4y^2 < (\rho-|B|^2)(4x+\rho-|B|^2)$.
\end{enumerate}
\end{enumerate}

\begin{remark}\label{Iwasawa}
{\rm Notice that $\frh_5$ is the Lie algebra underlying the Iwasawa manifold. The bi-invariant complex structure $J_0$
on $\frh_5$ corresponds to $\rho=1$ and $\epsilon=A=B=C=D=0$ in \eqref{nilpotentJ}. It is proved in \cite{FWW}
that the natural balanced metric on the Iwasawa manifold has $\gamma_1>0$, and by Lemma~\ref{partialOmega}~(i) we see that
the same holds for any other invariant $J_0$-Hermitian structure $\Omega$ on the Iwasawa manifold.
}
\end{remark}

\begin{proposition}\label{6dim-nilm-Gaud}
Let $(M,J)$ be a complex nilmanifold  of real dimension $6$ endowed with an invariant complex structure $J$.
Then, there is an invariant $J$-Hermitian metric $\Omega$ such that $\gamma_1(\Omega)<0$ if and only if $M$ corresponds to
$\frh_2,\frh_3,\frh_4$ or $\frh_5$.
\end{proposition}

\begin{proof}
If $J$ is bi-invariant then
$\gamma_1 (\Omega)$ cannot be negative, for any invariant $\Omega$.
By the previous discussion $J$ must be nilpotent and there exists a basis of invariant $(1,0)$-forms satisfying the reduced equations \eqref{epsilonzero-red-new}. From Lemma \ref{partialOmega} (i) it follows that there is an invariant $J$-Hermitian structure $\Omega$
such that $\gamma_1(\Omega)<0$ if and only if
\begin{equation} \label{ineqx}
2x>\rho+|B|^2.
\end{equation}
Since this implies $x>0$, the Lie algebras $\frh_6$ and $\frh_8$ are excluded because they correspond to cases
(a4) and (a5) above. Given $\rho$ and $B$, we can choose $x, y \in \mathbb{R}$ satisfying
\eqref{ineqx} and one of the conditions  (b1), (b2) or (b3). This shows  the  existence of a $\Omega$ with $\gamma_1 (\Omega) < 0$
on  the Lie algebras $\frh_2,\frh_4$ and $\frh_5$. Finally, the result for $\frh_3$ follows from (a2).
\end{proof}

Given a complex nilmanifold $(M,J)$ with invariant $J$, the existence of an SKT or a balanced Hermitian
metric implies the existence of an invariant one obtained
by the symmetrization process \cite{FG}. From Propositions~\ref{curve} and~\ref{6dim-nilm-Gaud} we conclude:

\begin{theorem}\label{existence-6dim-nilm}
Let $M$ be a nilmanifold  of real dimension $6$ with underlying Lie algebra isomorphic to  $\frh_2,\frh_3,\frh_4$ or $\frh_5$.
Then, there is an invariant complex structure $J$ on $M$ admitting a $($non invariant$)$ $1$-st Gauduchon metric and having no $J$-Hermitian SKT metrics.
\end{theorem}

The following example gives a deformation of an SKT structure   in a family of complex structures which do not admit any compatible SKT metric, but they have compatible $1$-st Gauduchon metrics.

\begin{example}\label{h4-deformation}
{\rm Let $M$ be a nilmanifold with underlying Lie algebra
$\frh_4$, and let $\{e^1,\ldots,e^6\}$
be a basis of invariant  1-forms on $M$ such that
$$
de^1=de^2=de^3=de^4=0,\quad de^5=e^{12},\quad de^6=e^{14}+e^{23}.
$$
We consider the family of complex structures $J_t$, $t\not=0$, defined by
$$
\omega^1=e^1+i\, e^4,\quad \omega^2=e^2+i\, t(e^3-e^4),\quad
\omega^3=2(e^5-i\, e^6).
$$
Since the complex equations for $J_t$ are
\begin{equation}\label{ecus-example}
d\omega^1=d\omega^2=0,\quad d\omega^3=\omega^{12}+\omega^{1\bar{1}}
+\omega^{1\bar{2}}+\frac{1}{t}\omega^{2\bar{2}},
\end{equation}
by {\rm (a3)} above the complex structure $J_t$ is defined on $M$ for any $t\not=0$. Moreover, by {\rm Lemma~\ref{partialOmega}~(i)}
the complex nilmanifold $(M,J_t)$ has a compatible SKT metric if and only if $t=1$, and admits an invariant Hermitian metric $\Omega$
with $\gamma_1(\Omega)<0$ if and only if $t<1$.

Therefore, $J_t$, $t\in (0,1]$, is a deformation of the complex structure
$J_1$ such that the complex manifold $(M,J_t)$
has a $1$-st Gauduchon metric for any $t$ but admits a compatible SKT metric only for $t=1$.

On the other hand, because of the form of the complex equations \eqref{ecus-example}, the fundamental form of any invariant $J_t$-Hermitian metric is equivalent to one given by \eqref{2forma} with $x_{1\bar{3}}=x_{2\bar{3}}=0$. Then, the balanced condition $d\Omega^2=0$ reduces to
$$
x_{1\bar{1}}+t\, x_{2\bar{2}} = t\, x_{1\bar{2}}.
$$
Since $x_{1\bar{1}}=i\, \lambda$ and $x_{2\bar{2}}=i\, \mu$, for some $\lambda,\mu>0$, we have that $x_{1\bar{2}}=i(\mu+\lambda/t)$. This implies that
$$
\det \left(\!\!\! \begin{array}{cc}
x_{1\bar{1}} & x_{1\bar{2}} \\
-\overline{x_{2\bar{1}}} & x_{2\bar{2}}
\end{array} \!\right)\! = (\mu+\lambda/t)^2 -\lambda\mu = \mu^2 +\frac{2-t}{t}\lambda\, \mu + \frac{\lambda^2}{t^2}
$$
is positive for any $t\in (0,1]$, which is a contradiction to the positive definiteness of the metric. By \cite{FG} we conclude that for any
$t\in (0,1]$ the complex manifold $(M,J_t)$ does not admit any balanced (invariant or not) metric.

}
\end{example}

\section{Products of Sasakian manifolds and circle bundles} \label{products}

Here we construct $1$-st Gauduchon metrics on products of Sasakian manifolds and on certain circle bundles over quasi-Sasakian manifolds.

We remind that an almost contact metric manifold  $(N^{2n - 1},
\varphi, \xi, \eta, g)$ is called {\em quasi-Sasakian} if it is normal and its
fundamental form $\Phi ( \cdot , \cdot) = g ( \varphi \cdot, \cdot )$  is closed.
If in particular $d\eta=\Phi$
then the almost contact metric  structure is said {\em Sasakian}.

Let $(M_1, \varphi_1, \xi_1, \eta_1, g_1)$ and $(M_2, \varphi_2,\xi_2, \eta_2, g_2)$ be two Sasakian manifolds of dimension $2n_1+1$ and $2n_2+1$, respectively.
On the product manifold $M=M_1\times M_2$ we consider the family of  complex structures  $J$ given by
Tsukada \cite{Tsukada}  and used in  \cite{M} to construct astheno-K\"ahler structures:
$$
J (X,Y)=\left( \varphi_1(X)-\frac{a}{b} \eta_1(X) \xi_1 -\frac{a^2+b^2}{b} \eta_2(Y) \xi_1,
\varphi_2(Y)+\frac{1}{b} \eta_1(X) \xi_2 +\frac{a}{b} \eta_2(Y) \xi_2 \right),
$$
for $X\in \mathfrak{X}(M_1)$ and $Y\in \mathfrak{X}(M_2)$, where $a, b \in \R$ and $b \neq 0$.

For any $t\in \mathbb{R}^*$  such that $ \frac{t}{b} > 0$ we consider on $M$ the $J$-Hermitian metric \begin{equation} \label{omprodSas} \Omega= \Phi_1+\Phi_2+t\eta_1\wedge\eta_2, \end{equation}
where $\Phi_j$ denotes the fundamental 2-form on $(M_j, \varphi_j, \xi_j, \eta_j, g_j)$ for each $j=1,2$.

\begin{theorem}\label{sasaki-prod}
Let us suppose that $2n=2(n_1+n_2+1)>6$. The Hermitian structure $\Omega$ given by \eqref{omprodSas}  on $M$ is $1$-st Gauduchon if and only if it is astheno-K\"ahler
if and only if $n_1(n_1-1)+2an_1n_2+(a^2+b^2)n_2(n_2-1)=0$.
\end{theorem}

\begin{proof}
Since
$d\Omega=t(\Phi_1\wedge\eta_2-\eta_1\wedge \Phi_2)$, we have that
$$d^c\Omega=JdJ\Omega=Jd\Omega=tJ(\Phi_1\wedge\eta_2-\eta_1\wedge \Phi_2).$$
The definition of the complex structure implies that $J\eta_1=\frac{a}{b}\eta_1+\frac{a^2+b^2}{b}\eta_2$ and $J\eta_2=-\frac{1}{b}\eta_1-\frac{a}{b}\eta_2$, which implies
$$
d^c\Omega=t(\Phi_1\wedge J\eta_2-J\eta_1\wedge \Phi_2)=-\frac{t}{b} [\Phi_1\wedge \eta_1 + a \Phi_1\wedge \eta_2 + a \eta_1\wedge \Phi_2
+(a^2+b^2)\eta_2\wedge \Phi_2].
$$
Thus, $dd^c\Omega=\frac{t}{b} [\Phi_1\wedge \Phi_1 + 2a \Phi_1\wedge \Phi_2 +(a^2+b^2)\Phi_2\wedge \Phi_2]$.

On the other hand,
$$
\begin{array}{lcl}
\Omega^{n-2} \zzz & = &\zzz \displaystyle\sum_{r=0}^{n-2} \left(\!\!\!\begin{array}{c}
n-2 \\
r
\end{array}\!\!\!\right) (\Phi_1+\Phi_2)^{n-2-r}\wedge (t\eta_1\wedge\eta_2)^r\\[10pt]
\zzz & = &\zzz (\Phi_1+\Phi_2)^{n-3}\wedge \left[ \Phi_1+\Phi_2+t(n-2)\eta_1\wedge\eta_2 \right]\\[4pt]
\zzz & = &\zzz \displaystyle\sum_{s=0}^{n-3} \left(\!\!\!\begin{array}{c}
n-3 \\
s
\end{array}\!\!\!\right) \Phi_1^{n-3-s}\wedge \Phi_2^s\wedge \left[ \Phi_1+\Phi_2+t(n-2)\eta_1\wedge\eta_2 \right],
\end{array}
$$
because $(\eta_1\wedge\eta_2)^r$ is non-zero only for $r=0,1$.
Therefore,
$$
\begin{array}{lcl}
dd^c\Omega\wedge\Omega^{n-2} \zzz & = &\zzz \frac{t}{b} \displaystyle\sum_{s=0}^{n-3} \left(\!\!\!\begin{array}{c}
n-3 \\
s
\end{array}\!\!\!\right) \left[ \Phi_1^2 + 2a \Phi_1\wedge \Phi_2 +(a^2+b^2)\Phi_2^2 \right] \wedge \\[12pt]

&  & \Phi_1^{n-3-s}\wedge \Phi_2^s\wedge \left[ \Phi_1+\Phi_2+t(n-2)\eta_1\wedge\eta_2 \right]\\[8pt]
\zzz & = &\zzz \frac{t}{b} \displaystyle\sum_{s=0}^{n-1} C(n,s) \Phi_1^{n-1-s}\wedge \Phi_2^s\wedge [\Phi_1+\Phi_2+t(n-2)\eta_1\wedge\eta_2],
\end{array}
$$
where $C(n,0)=1$, $C(n,1)=n-3+2a$, $C(n,n-2)=2a+(a^2+b^2)(n-3)$, $C(n,n-1)=a^2+b^2$, and
$$C(n,s)=\left(\!\!\!\begin{array}{c}
n-3 \\
s
\end{array}\!\!\!\right) +2a \left(\!\!\!\begin{array}{c}
n-3 \\
s-1
\end{array}\!\!\!\right) + (a^2+b^2) \left(\!\!\!\begin{array}{c}
n-3 \\
s-2
\end{array}\!\!\!\right),
$$
for $2\leq s\leq n-3$.

Now, by the same argument as in \cite{M} we get that
$$
\begin{array}{lcl}
dd^c\Omega\wedge\Omega^{n-2} \zzz & = &\zzz \frac{t}{b} C(n,n_2) \Phi_1^{n_1}\wedge \Phi_2^{n_2}\wedge \left[ \Phi_1+\Phi_2+t(n-2)\eta_1\wedge\eta_2 \right]\\[4pt]
\zzz & = &\zzz \frac{t^2}{b}(n-2) C(n,n_2) \Phi_1^{n_1}\wedge \Phi_2^{n_2}\wedge\eta_1\wedge\eta_2.
\end{array}
$$
It is clear that $\Omega$ is $1$-st Gauduchon if and only if $C(n,n_2)=0$
which is equivalent to  $n_1(n_1-1)+2an_1n_2+(a^2+b^2)n_2(n_2-1)=0.$ The latter condition is also equivalent to $\Omega$ be astheno-K\"ahler by \cite[Theorem 4.1]{M}.
\end{proof}

\begin{remark}\label{sign-in-product}
{\rm By a similar calculation as in the proof of the previous theorem, for $n>3$ we get that
$$\Omega^n=n t \, \left [  \left ( \begin{array}{c} n - 3\\ n_2 \end{array} \right ) + 2  \left ( \begin{array}{c} n - 3\\ n_2 - 1 \end{array} \right ) + \left ( \begin{array}{c} n - 3\\ n_2 - 2 \end{array} \right )\right ]\Phi_1^{n_1}\wedge \Phi_2^{n_2}\wedge\eta_1\wedge\eta_2$$
and using $dd^c\Omega=-2i\partial\db\Omega$ we conclude that
if the metric $\Omega$ is not $1$-st Gauduchon then
$$
\frac{i}{2}\partial\db\Omega \wedge\Omega^{n-2}=\frac{(n-2) t}{n b}  \left [   \frac{ n_1 (n_1 - 1) + 2 n_1 n_2 a + n_2 (n_2 - 1) (a^2 + b^2)} { n_1 (n_1 - 1) + 2 n_1 n_2 + n_2 (n_2 - 1) }  \right ] \Omega^n.
$$
}
\end{remark}

\begin{proposition}\label{sasaki-prod-dim6}
Let us suppose that $n_1 = n_2 =3$, that is, the real dimension of $M$ is equal to $2n=6$.
Then, the Hermitian metric $\Omega$ given by \eqref{omprodSas}  is $1$-st Gauduchon if and only if it is SKT if and only if $a=0$.
Moreover, $\gamma_1(\Omega)<0$ if and only if $a<0$, which implies the existence of a $1$-st Gauduchon metric
on $(M, J_a)$ for any $a<0$ whenever the manifold is compact.
\end{proposition}

\begin{proof}
Since $M_j$ is 3-dimensional, $dd^c\Omega=\frac{2at}{b}\Phi_1\wedge \Phi_2$ and
$$dd^c\Omega\wedge \Omega= \frac{2at^2}{b}  \Phi_1\wedge \Phi_2\wedge \eta_1\wedge\eta_2.$$ Since $\Omega^3=6t \Phi_1\wedge \Phi_2\wedge \eta_1\wedge\eta_2$
and $dd^c\Omega=-2i\partial\db\Omega$, we conclude that
$$\frac{i}{2}\partial\db\Omega \wedge\Omega=\frac{a t}{3 b} \Omega^3.$$ The last assertion follows from Proposition~\ref{curve} and Corollary 10 in \cite{FWW}.
\end{proof}

Next we construct $1$-st Gauduchon metrics on circle bundles over certain almost contact metric manifolds.
Let $(N, \varphi, \xi, \eta, g)$ be a $(2n-1)$-dimensional almost contact metric manifold and
let $F$ be a closed $2$-form on $N$ which represents an
integral cohomology class on $N$. From the well known result of
Kobayashi \cite{Kob}, we can consider the circle bundle $S^1
\hookrightarrow P \to N$, with connection $1$-form $\theta$ on $P$
whose curvature form is $d\theta = \pi^*(F)$, where $\pi: P \to
N$ is the projection.

By using a normal almost contact structure $(\varphi, \xi, \eta)$ on $N$ and a  connection
$1$-form $\theta$ on $P$ such that $$ F  (\varphi X, Y) + F (X,  \varphi Y) =0, \quad  F (\xi, X)=0, \quad \forall X, Y \in \chi (N), $$ one can define  a complex structure $J$ on $P$ as
follows (see \cite{Og}). For any
right-invariant vector field $X$ on $P$,
$J X$ is given by
\begin{equation} \label{acxS1}
\theta (JX) = - \pi^*( \eta (\pi_* X)),\quad\
\pi_* (JX) = \varphi (\pi_* X)  + \tilde \theta (X) \xi,
\end{equation}
where $\tilde \theta (X)$ is the unique function on $N$ such that
$\pi^* \tilde  \theta (X) =  \theta (X)$.

The above definition can be extended  to arbitrary vector fields $X$
on $P$, since $X$ can be written in the form
$
X = \sum_j f_j X_j,
$
with $f_j$ smooth functions on $P$ and $X_j$ right-invariant vector
fields. Then $JX = \sum_j f_j  J X_j$.

Moreover,   a Riemannian metric  $h$ on $P$ compatible
with $J$ (see \cite{Og}) is given by
\begin{equation} \label{acxS2-1}
\begin{array} {l}
h(X, Y) = \pi^* g( \pi_* X,  \pi_* Y) +  \theta(X)  \theta(Y),
\end{array}
\end{equation}
for any right-invariant vector fields $X, Y$.

If $(N, \varphi, \xi, ,\eta, g)$ is  quasi- Sasakian it was shown if \cite{FFUV}  that $h$ is SKT if and only if
$$
d \eta \wedge d \eta + F \wedge F =0.
$$
For $1$-st Gauduchon metrics we prove the following

\begin{theorem}\label{non-t-circle-bundle}
Let $(N, \varphi, \xi, \eta, g)$ be a $(2n-1)$-dimensional
quasi-Sasakian manifold, $n>2$, and let $F$
be a closed $2$-form on
$N$ which represents an integral cohomology class. Consider
the circle bundle $S^1 \hookrightarrow P \stackrel{\pi}{\longrightarrow} N$ with
connection $1$-form $\theta$ whose curvature form is $d\theta =
\pi^*(F)$. If $d\theta$ is $J$-invariant
then the almost Hermitian structure $(J,h)$ on $P$, defined by
\eqref{acxS1} and \eqref{acxS2-1}, is $1$-st Gauduchon if and only if
\begin{equation}\label{bundle-1-Gauduchon}
(d \eta \wedge d\eta + F\wedge F) \wedge \Phi^{n-3}=0,
\end{equation}
where $\Phi$ denotes the fundamental form of the quasi-Sasakian structure $(\varphi, \xi, \eta, g)$.
\end{theorem}

\begin{proof}
Since $\Omega=\pi^*\Phi + \pi^*\eta\wedge \theta$ is the fundamental $2$-form associated to $(J, h)$ and the structure is quasi-Sasakian we have
$$
Jd\Omega= J(\pi^*(d\eta))\wedge J\theta - J(\pi^*(\eta))\wedge J(d\theta).
$$
Since $J\theta= \pi^*(\eta)$ and   the  almost contact structure is normal we have
$$J(\pi^*(d\eta))=\pi^*(d\eta), \quad d\eta (\xi, X ) =0, \, \forall X \in \chi(N)
$$
and therefore
$$
\begin{array}{lcl}
dJd\Omega \zzz & = &\zzz d(\pi^*(d\eta)\wedge \pi^*(\eta) + \theta\wedge J(\pi^*(F)))\\[6pt]
\zzz & = &\zzz  \pi^*(d\eta)\wedge \pi^*(d\eta) + d\theta\wedge J(\pi^*(F)) - \theta\wedge d[J(\pi^*(F))]\\[6pt]
\zzz & = &\zzz  \pi^*(d\eta)\wedge \pi^*(d\eta) + \pi^*(F)\wedge \pi^*(F),
\end{array}
$$
where in the last equality we used that $d\theta=\pi^*(F)$ is invariant by the complex structure $J$. Therefore,
$$
\begin{array}{lcl}
dJd\Omega\wedge \Omega^{n-2} \zzz & = &\zzz
\pi^*(d\eta\wedge d\eta + F\wedge F )\wedge (\pi^*\Phi + \pi^*\eta\wedge \theta)^{n-2}\\[6pt]
\zzz & = &\zzz \pi^*(d\eta\wedge d\eta + F\wedge F)\wedge \left[ \pi^*\Phi^{n-2} + (n-2) \pi^*(\Phi^{n-3} \wedge \eta)\wedge \theta \right] \\[6pt]
\zzz & = &\zzz (n -2) \pi^* \left  [ (d\eta\wedge d\eta + F\wedge F) \wedge  \Phi^{n-3} \wedge \eta \right ] \wedge \theta
\end{array}
$$
and $dJd\Omega\wedge \Omega^{n-2}$ vanishes if and only if the equation \eqref{bundle-1-Gauduchon} is satisfied.
\end{proof}

\begin{example} {\rm Consider the $5$-dimensional    solvable Lie group $S$  with structure equations
$$
\left \{ \begin{array}{l} d e^i =0,  \,  i = 1, 4,\\[4 pt] d e^2 = e^{13},\\ d e^3 = - e^{12},\\ d e^5 = e^{14} + e^{23}, \end{array} \right.
$$
endowed with the left-invariant Sasakian structure $(\varphi, \xi, \eta, g)$ given by
$$
\varphi (e_1) =  e_4, \quad \varphi (e_2) = - e_3, \quad \eta = e^5, \quad g = \sum_{i = 1}^5 (e^i)^2.
$$
By \cite[Corollary 4.2]{AFV} the solvable Lie group $S$ admits a compact quotient by a uniform discrete subgroup $\Gamma$. If we consider   the $S^1$-bundle $(P, J)$ over  the compact quotient $\Gamma \backslash S$  with connection $1$-form $\theta$ such that $$d \theta =   \pi^* ( 2 e^{14} - 2 e^{23})$$ we have that $\Omega =  \phi^* \Phi + \pi^* \eta \wedge \theta$ is a $J$-Hermitian structure on the $S^1$-bundle over  $\Gamma \backslash S$  such that $\gamma_1 (\Omega) <0$. Therefore,   by Proposition  \ref{curve}  the complex manifold  $(P, J)$ admits a $1$-st Gauduchon metric.}
\end{example}

\begin{remark}
{\rm As a consequence of Theorem \ref{non-t-circle-bundle} we have that on a trivial $S^1$-bundle
over a Sasakian manifold the metric $h$  compatible with the complex structure
$J$ given by \eqref{acxS1} cannot be $1$-st Gauduchon.}
\end{remark}

We show next  that in real dimension 8 there are non-SKT complex nilmanifolds having
invariant $1$-st Gauduchon metrics. The complex equations
\begin{equation}\label{A-family}
d\omega^1=d\omega^2=d\omega^3=0, \quad d\omega^4= A\, \omega^{1\bar{1}} - \omega^{2\bar{2}} - \omega^{3\bar{3}}
\end{equation}
define a complex 8-dimensional nilmanifold $(M_A,J_A)$ for any $A=p+i\, q \in \mathbb{C}$. It is easy to see that $M_A$
is a quotient of $H_3\times H_5$ if $q \not=0$,
a quotient of $H_7\times \mathbb{R}$ if $q=0$, and a quotient of $H_5\times \mathbb{R}^3$ if $p=q=0$.  Here we denote by $H_{2n+1}$
the generalized real Heisenberg group of dimension $2n + 1$.

\begin{lemma}\label{equiv}
Let $M$ be a nilmanifold of real dimension $2n$ endowed with an invariant complex structure $J$, and let
$\Omega$ be any invariant $J$-Hermitian structure. For $k=1,\ldots,\left[\frac{n}{2}\right]-1$, the structure
$\Omega$ is $k$-th Gauduchon if and only if it is $(n-k-1)$-th Gauduchon.
\end{lemma}

\begin{proof}
For any $k=1,\ldots,\left[\frac{n}{2}\right]-1$, we use the formula
\begin{equation} \label{formulakandn-k-1}
\int_M  \partial \overline\partial \Omega^k \wedge \Omega^{n-k-1} =
\int_M \partial \overline\partial \Omega^{n-k-1} \wedge \Omega^k,
\end{equation}
which holds for a general compact  complex  manifold. Indeed,  $$\partial \overline\partial \Omega^k \wedge \Omega^{n-k-1}
=
 d (\overline\partial \Omega^k \wedge \Omega^{n-k-1})
 + \overline\partial \Omega^k \wedge \partial \Omega^{n-k-1}$$
and by Stokes theorem
\begin{equation}\label{formula1}
\int_M \partial \overline\partial \Omega^k \wedge \Omega^{n-k-1}
= \int_M  \overline\partial \Omega^k \wedge \partial \Omega^{n-k-1}.
\end{equation}
On the other hand, in a similar way  we get
\begin{equation} \label{formula2}
\int_M  \partial \overline\partial \Omega^{n-k-1} \wedge \Omega^k
= \int_M \overline\partial \Omega^k \wedge \partial \Omega^{n-k-1}.
\end{equation}
Comparing  the two formulas \eqref{formula1} and \eqref{formula2} we obtain \eqref{formulakandn-k-1}.

Since $\Omega$ is invariant on the complex nilmaniold $(M, J)$,
there exist $\lambda, \mu\in \mathbb{R}$ such that $\frac{i}{2}\partial\overline\partial \Omega^k \wedge \Omega^{n-k-1}=\lambda\, \Omega^n$ and
$\frac{i}{2}\partial\overline\partial \Omega^{n-k-1} \wedge \Omega^k=\mu\, \Omega^n$.
Therefore,
$$
\lambda \int_M \Omega^n= \frac{i}{2} \int_M \partial \overline\partial \Omega^k \wedge \Omega^{n-k-1}=
\frac{i}{2} \int_M \partial \overline\partial \Omega^{n-k-1} \wedge \Omega^k= \mu \int_M \Omega^n,
$$
which implies that $\lambda=0$ if and only if $\mu=0$, i.e. $\Omega$ is $k$-th Gauduchon if and only if it is $(n-k-1)$-th Gauduchon.
\end{proof}

\begin{proposition}\label{8dim-nilm-family}
Let $(M_A,J_A)$ be the complex nilmanifold given by \eqref{A-family} and let
$$
\Omega= \sum_{j,k=1}^4 x_{j\bar{k}}\, \omega^{j\bar{k}},
$$
where $x_{j\bar{k}} \in \mathbb{C}$ and $\overline{x_{k\bar{j}}}=-x_{j\bar{k}}$, be the
fundamental 2-form of any invariant
$J_A$-Hermitian metric $g$ on $M_A$.
Then, $\Omega$ is never SKT, and $\Omega$ is $1$-st Gauduchon if and only if it is $2$-nd Gauduchon if and only if it is astheno-K\"ahler.

Moreover,
\begin{enumerate}
\item[(i)] $\Omega$ is $1$-st Gauduchon if and only if
$$
p\, \det \left(\!\!\! \begin{array}{cc}
x_{2\bar{2}} & x_{2\bar{4}} \\
-\overline{x_{2\bar{4}}} & x_{4\bar{4}}
\end{array} \!\right)\!
+ p\, \det \left(\!\!\! \begin{array}{cc}
x_{3\bar{3}} & x_{3\bar{4}} \\
-\overline{x_{3\bar{4}}} & x_{4\bar{4}}
\end{array} \!\right)\!
= \det \left(\!\!\! \begin{array}{cc}
x_{1\bar{1}} & x_{1\bar{4}} \\
-\overline{x_{1\bar{4}}} & x_{4\bar{4}}
\end{array} \!\right)\! ;
$$
\item[(ii)] $\Omega$ is balanced if and only if $q=0$ and
$$
p\, \det \left(\!\!\! \begin{array}{ccc}
x_{2\bar{2}} & x_{2\bar{3}} & x_{2\bar{4}} \\
-\overline{x_{2\bar{3}}} & x_{3\bar{3}} & x_{3\bar{4}} \\
-\overline{x_{2\bar{4}}} & -\overline{x_{3\bar{4}}} & x_{4\bar{4}}
\end{array} \!\right)\!
=
\det \left(\!\!\! \begin{array}{ccc}
x_{1\bar{1}} & x_{1\bar{2}} & x_{1\bar{4}} \\
-\overline{x_{1\bar{2}}} & x_{2\bar{2}} & x_{2\bar{4}} \\
-\overline{x_{1\bar{4}}} & -\overline{x_{2\bar{4}}} & x_{4\bar{4}}
\end{array} \!\right)\!
+
\det \left(\!\!\! \begin{array}{ccc}
x_{1\bar{1}} & x_{1\bar{3}} & x_{1\bar{4}} \\
-\overline{x_{1\bar{3}}} & x_{3\bar{3}} & x_{3\bar{4}} \\
-\overline{x_{1\bar{4}}} & -\overline{x_{3\bar{4}}} & x_{4\bar{4}}
\end{array} \!\right)\! .
$$
\end{enumerate}

In particular, if $p<0$ then the complex manifold $(M_A,J_A)$ does not admit either invariant balanced metrics or $1$-st Gauduchon metrics.
\end{proposition}

\begin{proof} From the complex equations \eqref{A-family} if follows that $\Omega$ satisfies
$$
\partial \db \Omega =  x_{4\bar{4}} (A+\bar{A}) \omega^{1\bar{1}2\bar{2}} + x_{4\bar{4}} (A+\bar{A}) \omega^{1\bar{1}3\bar{3}}
- 2 x_{4\bar{4}} \omega^{2\bar{2}3\bar{3}},
$$
which never vanishes because $x_{4\bar{4}}\not=0$. Therefore, any invariant $J_A$-Hermitian metric is not SKT.

Moreover, a direct calculation shows that
$$
\begin{array}{lcl}
\partial \db \Omega \wedge \Omega^2 &=&2 x_{4\bar{4}}\left[ (A+\bar{A}) (x_{2\bar{2}}x_{4\bar{4}}+x_{3\bar{3}}x_{4\bar{4}}+|x_{2\bar{4}}|^2 +|x_{3\bar{4}}|^2) \right .\\[5 pt]
& & \left . -2 (x_{1\bar{1}}x_{4\bar{4}}+|x_{1\bar{4}}|^2)
\right] \omega^{1\bar{1}2\bar{2}3\bar{3}4\bar{4}}
\end{array}
$$
and therefore the metric is $1$-st Gauduchon if and only if
\begin{equation}\label{gaud-dim8}
(A+\bar{A}) (x_{2\bar{2}}x_{4\bar{4}}+x_{3\bar{3}}x_{4\bar{4}}+|x_{2\bar{4}}|^2 +|x_{3\bar{4}}|^2) -2 (x_{1\bar{1}}x_{4\bar{4}}+|x_{1\bar{4}}|^2)=0.
\end{equation}
It is easy to check that the condition for astheno-K\"ahler 
is satisfied if and only if \eqref{gaud-dim8} holds.
This condition is precisely the one given in (i), so the proof of (i) follows from Lemma~\ref{equiv}.

The balanced condition for $\Omega$, i.e. $d\Omega^3=0$, can be seen to be equivalent to the system given by the equation
$$
A\, \det \left(\!\!\! \begin{array}{ccc}
x_{2\bar{2}} & x_{2\bar{3}} & x_{2\bar{4}} \\
-\overline{x_{2\bar{3}}} & x_{3\bar{3}} & x_{3\bar{4}} \\
-\overline{x_{2\bar{4}}} & -\overline{x_{3\bar{4}}} & x_{4\bar{4}}
\end{array} \!\right)\!
=
\det \left(\!\!\! \begin{array}{ccc}
x_{1\bar{1}} & x_{1\bar{2}} & x_{1\bar{4}} \\
-\overline{x_{1\bar{2}}} & x_{2\bar{2}} & x_{2\bar{4}} \\
-\overline{x_{1\bar{4}}} & -\overline{x_{2\bar{4}}} & x_{4\bar{4}}
\end{array} \!\right)\!
+
\det \left(\!\!\! \begin{array}{ccc}
x_{1\bar{1}} & x_{1\bar{3}} & x_{1\bar{4}} \\
-\overline{x_{1\bar{3}}} & x_{3\bar{3}} & x_{3\bar{4}} \\
-\overline{x_{1\bar{4}}} & -\overline{x_{3\bar{4}}} & x_{4\bar{4}}
\end{array} \!\right)\!
$$
and its conjugate. The positive definiteness of the metric $g$ implies in particular that
$$-i\, x_{j\bar{j}}, \quad
- \det \left(\!\!\! \begin{array}{cc}
x_{j\bar{j}} & x_{j\bar{k}} \\
-\overline{x_{j\bar{k}}} & x_{k\bar{k}}
\end{array} \!\right),
\quad
i\, \det \left(\!\!\! \begin{array}{ccc}
x_{j\bar{j}} & x_{j\bar{k}} & x_{j\bar{l}} \\
-\overline{x_{j\bar{k}}} & x_{k\bar{k}} & x_{k\bar{l}} \\
-\overline{x_{j\bar{l}}} & -\overline{x_{k\bar{l}}} & x_{l\bar{l}}
\end{array} \!\right)
$$
are strictly positive real numbers for $1\leq j<k<l\leq 4$. Therefore, the balanced condition for $\Omega$ is satisfied
if and only if $A$ is real and (ii) holds.

Finally, the last assertion in the proposition follows easily from the positive definiteness of the metric $g$.
\end{proof}

\begin{remark}
{\rm The family \eqref{A-family} belongs to a more general family of astheno-K\"ahler manifolds given in \cite{FT}.
}
\end{remark}

As a consequence of the above proposition and \cite{FG} we get that the complex manifold $(M_A,J_A)$ does not admit any SKT Hermitian structure, and if  $p<0$ or $q\not=0$ then $(M_A,J_A)$ does not possess any balanced Hermitian metric.

\begin{corollary}\label{h7xR}
Let $M$ be a compact quotient of $H_7\times \mathbb{R}$ by a lattice of maximal rank. Then, there exists a complex structure on $M$ which
does not admit any SKT Hermitian structure, but having balanced Hermitian metrics and $k$-th Gauduchon metrics for $k=1$ and $2$.
\end{corollary}

Notice that a result similar to the above corollary holds for $H_5\times \mathbb{R}^3$, which corresponds to the case $A=0$.

\begin{corollary}\label{h3xh5}
Let $M$ be a compact quotient of $H_3\times H_5$ by a lattice of maximal rank. Then, there exists a complex structure on $M$ which
does not admit neither SKT nor balanced Hermitian structure, but having invariant $k$-th Gauduchon metrics for $k=1$ and $2$.
\end{corollary}

\begin{remark}
{\rm Note that for $p = q = 0$ one has on $M_A$ the complex structure coming from the  product  $H_5 \times \R^3$ of a Sasakian manifold with a cosymplectic one.
If $p  < 0$ (respectively  $p  > 0$)  and $q =0$ then
the complex structure comes from
a product of  a Sasakian (resp. quasi-Sasakian) manifold with $\R$.
Finally,
for $q \neq 0$ the complex structure comes from a product  $H_5 \times H_3$ of  two quasi-Sasakian manifolds.
}
\end{remark}

\section{Blow-ups} \label{blow-ups}

In \cite{FT2} it was shown that the SKT condition is preserved by the blow-up  of the complex manifold at a point  or along a compact submanifold. We can  prove the following

\begin{proposition} Let $(M,J)$ be a  complex manifold of complex dimension $n >2$  and endowed with a  $J$-Hermitian metric $g$ such that  its fundamental $2$-form $\Omega$ satisfies the condition $\partial \overline \partial \Omega \wedge \Omega^{n -2}  > 0$ (or $< 0)$. Then, the blow-up $\tilde M_p$ at a point $p \in M$ and the blow-up $\tilde M_Y$ along a compact submanifold $Y$ admit a Hermitian metric such that
its fundamental $2$-form $\tilde \Omega$ satisfies the 
condition $\partial \overline \partial  \tilde \Omega \wedge \tilde \Omega^{n -2}  > 0$ (or $< 0)$.
\end{proposition}

\begin{proof}
Let us start the proof in the case of the blow-up of $M$ at a point $p$.
We recall now  briefly the construction of the blow-up.
Let $z = (z_1, \ldots ,z_n)$ be holomorphic coordinates in an open
set $U$ centered around the point $p \in M$. The
blow-up $\tilde M_p$ of $M$ is the complex manifold obtained by
adjoining to $M \setminus \{ p \}$ the manifold
$$
\tilde U = \{ (z, l) \in U \times \CP^{n - 1} \, \vert \, z \in l
\}
$$
by using the isomorphism
$$
\tilde U \setminus \{ z =0 \} \cong U \setminus \{ p \}
$$
given by the projection $ (z, l) \to z$. In this way there is a
natural projection $\pi: \tilde M_p \to M$ extending the identity
on $M \setminus \{ p \}$ and the exceptional divisor $\pi^{-1}(p)$
of the blow-up is naturally isomorphic to the complex projective
space $\CP^{n - 1}$.

The $2$-form $\pi^* \Omega$ is  a $(1,1)$ form on $\tilde M_p$  since $\pi$ is holomorphic, but it is
not positive definite on $\pi^{-1} (M \setminus \{ p \})$. As in
the K\"ahler case, let $h$ be a ${\mathcal C}^{\infty}$-function
having support in $U$, i.e. $0 \leq h \leq 1$ and $h=1$ in a
neighborhood of $p$. On $U \times (\C^n \setminus \{ 0 \})$
consider the $2$-form
$$
\gamma = i \partial \overline \partial \left( (p_1^* h) p_2^* \log
\vert \vert \cdot\vert \vert^2 \right)\,,
$$
where $p_1$ and $p_2$ denote the two projections of $U \times
(\C^n \setminus \{ 0 \})$ on $U$ and $\C^n\setminus\{0\}$,
respectively.
\newline
Let $\psi$ be the restriction of $\gamma$ to
$\tilde M_p$. Then there exists a small enough real number
$\epsilon$ such that the $2$-form $\tilde \Omega = \epsilon \psi +
\pi^* \Omega$ is positive definite.

Then
$$
\partial \overline \partial  \tilde \Omega \wedge \tilde \Omega^{n -2} =  \partial \overline \partial  \pi^* \Omega \wedge (\pi^* \Omega + \epsilon \psi)^{n -2}.
$$
Since
$\partial \overline \partial  \pi^* \Omega \wedge  \pi^* \Omega^{n -2} >0$ (or $< 0$)  by assumption we have that we can choose $\epsilon$
small enough
in order to  have $\partial \overline \partial  \tilde \Omega \wedge \tilde \Omega^{n -2} >0$
if  $\partial \overline \partial  \Omega \wedge \Omega^{n -2} >0$
(or $< 0$  if  $\partial \overline \partial  \Omega \wedge \Omega^{n -2}  <0$).

For the case of  the blow-up  of  $(M, J)$  along a compact submanifold  one proceeds  in a similar way.
Let $Y\subset M$ be a compact complex submanifold of $M$. Then consider
the blow-up $\tilde{M}_Y$ of $M$ along $Y$. If we denote by  $\pi :\tilde{M}_Y\to M$ the holomorphic projection, by
construction $\pi:\tilde{M}\setminus\pi^{-1}(Y)\to M\setminus Y$
is a biholomorphism and $\pi^{-1}(Y)\cong\Proj(\mathcal{N}_{Y\vert
M})$, where $\Proj(\mathcal{N}_{Y\vert M})$ is the projectified of
the normal bundle of $Y$.  As in \cite{FT2}  we can show that there exists a holomorphic line
bundle $L$ on $\tilde{M}_Y$ such that $L$ is trivial on
$\tilde{M}_Y\setminus\pi^{-1}(Y)$ and such that its restriction to
$\pi^{-1}(Y)$ is isomorphic to
$\mathcal{O}_{\Proj(\mathcal{N}_{Y\vert M})}(1)$. \newline Let $h$
be a Hermitian structure on
$\mathcal{O}_{\Proj(\mathcal{N}_{Y\vert M})}(1)$ and $\omega$ be
the corresponding Chern form. As in \cite{FT2} one can show that the metric $h$ can
be extended to a metric structure $\hat{h}$ on $L$,
in such a way
that $\hat{h}$ is the flat metric structure on the complement of a
compact neighborhood $W$ of $Y$ induced by the trivialization of
$L$ on $\tilde{M}_Y\setminus\pi^{-1}(Y)$. Therefore, the Chern
curvature $\hat{\omega}$ of $L$ vanishes on $M\setminus W$ and
$\hat{\omega}\vert_{\Proj(\mathcal{N}_{Y\vert M})}=\omega$.
\newline Hence, since $Y$ is compact, there exists $\epsilon
\in\R,\,\epsilon >0$, small enough such that
$$
\tilde{\Omega} =\pi^* \Omega +\epsilon\,\hat{\omega}
$$
is positive definite.  As for the previous case of
blow-up at a point we can  choose $\epsilon$ small enough in order
to have $\partial \overline \partial  \tilde \Omega \wedge \tilde \Omega^{n -2} >0$
if  $\partial \overline \partial  \Omega \wedge \Omega^{n -2} >0$
(or $< 0$  if  $\partial \overline \partial  \Omega \wedge \Omega^{n -2}  <0$).
\end{proof}

If one applies the previous proposition to a compact complex manifold of complex dimension $3$
endowed with a Hermitian structure $(J, g, \Omega)$ such that $\gamma_1 (\Omega) <0$,
then by using Proposition~\ref{curve} and Corollary 10 in \cite{FWW}, there exists a $1$-st Gauduchon
metric on the complex blow-up at a point or along a compact submanifold.

\section{Twists}

In this section  we will show that by applying the twist construction of  \cite[Proposition 4.5]{Sw}
one can get new  simply-connected $6$-dimensional  complex manifolds which admit $1$-st Gauduchon (non SKT) metrics.

We recall that in general, given a manifold $M$ with
a $T_M$-torus action and a principal $T_P$-torus bundle $P$ with connection $\theta$, if the torus
action of $T_M$  lifts to $P$ commuting with the principal action of $T_P$, then one may
construct the twist $W$ of the manifold, as the quotient of $P/T_M$ by the torus action
(see \cite{Sw}). Moreover, if the lifted torus action preserves the principal connection
$\theta$, then tensors on $M$ can be transferred to tensors on $W$ if their pullbacks to
$P$ coincide on $\mathcal H = {\mbox {Ker}} \, \theta$. A differential form $\alpha$ on
$M$ is ${\mathcal H}$-{\em related} to a differential form $\alpha_W$ on $W$,
$\alpha \sim_{\mathcal H} \alpha_W$, if their pull-backs to $P$ coincide on $\mathcal H$.

Let now $A_M  \cong A^m $  be a connected  abelian  group acting on $M$ in such a way
that there is a smooth twist $W$ given via  curvature  $F \in \Omega^2 (M, {\mathfrak  a}_P)$
and invertible lifting function $a \in \Omega^0 (M, {\mathfrak a}_M \otimes {\mathfrak a}_P^*)$,
where  ${\mathfrak  a}_M$  (respectively ${\mathfrak  a}_P$)  denotes the Lie algebra of  $A_M $ (respectively $A_P \cong A^m$).

If the twist $W$ has a Hermitian structure $(J_W, g_W, \Omega_W)$ induced by the one $(J, g, \Omega)$ on $M$, one has
for the fundamental $2$-forms and for the torsion $3$-forms of the Bismut connections the following relations
$$
\Omega_W   \sim_{\mathcal H} \Omega, \quad c_W   \sim_{\mathcal H} c - a^{-1
} J_M  F \wedge \xi^{b},
$$
where $\xi: {\mathfrak a}_M \to \chi (M)$ is the infinitesimal action and  $\xi^b$ is the dual of $\xi$ by using the metric.
In the case $F$ is of instanton type, i.e.  if $F$ is of type $(1,1)$, one has that
$$
d c_W \sim_{\mathcal H} dc - a^{-1} F \wedge [i_{\xi} c + d \xi^b - g(\xi, \xi) a^{-1} F].
$$
Swann obtained with the instanton construction new examples of simply-connected SKT manifolds.
One can adapt the  previous construction in the case of  complex manifolds endowed  with a $1$-st Gauduchon metric.

Let $(N^{2n}, J)$ be a  simply-connected compact
complex manifold of complex dimension $n \geq 2$ with a $J$-Hermitian metric $g_{N^{2n}}$ which
is $1$-st Gauduchon.
Consider the product $M^{2n+2} = N^{2n} \times \T^2$, where $\T^2$ is a $2$-torus with an invariant K\"ahler structure.
Then $M^{2n+2}$ has a $1$-st Gauduchon metric $g_{M^{2n+2}} = g_{N^{2n}} + g_{\T^2}$
with torsion $3$-form $c$ supported on $N^{2n}$.
\newline
Let $\xi$ be
the torus action on the $\T^2$ factor. We have then ${\mathfrak a}_M \cong {\mathfrak a}_P \cong \R^2$,
 $$
 i_{\xi} c =0, \quad d \xi^b =0
 $$
 and $a$ is a constant isomorphism ${\mathfrak a}_M \to {\mathfrak a}_P$.
 Since
 $$
 \Omega_W^{n - 1}  \sim_{\mathcal H} (\Omega_{N^{2n}} + \Omega_{\T^2})^{n -1}  =   \Omega_{N^{2n  }}^{n -1} + (n -1)  \Omega_{N^{2n}}^{n -2}  \wedge \Omega_{\T^2}
 $$
 we get
 $$
(d c_W  \wedge   \Omega_W^{n - 1} ) \sim_{\mathcal H}  [dc + g(\xi, \xi) a^{-1} F \wedge  a^{-1} F] \wedge  [\Omega_{N^{2n}}^{n -1} + (n -1)  \Omega_{N^{2n - 2}}^{n -2}  \wedge \Omega_{\T^2}
]
$$
 and thus by using the assumption that $g_{N^{2n}}$ is $1$-st Gauduchon we obtain
 $$
(d c_W  \wedge   \Omega_W^{n - 1} ) \sim_{\mathcal H}  g(\xi, \xi) a^{-1} F \wedge  a^{-1} F \wedge   (n -2)  \Omega_{N^{2n }}^{n -2}  \wedge \Omega_{\T^2}.
$$
Assume that there are two linearly independent integral closed $(1,1)$-forms $F_i \in \Lambda^{1,1}_{\Z} (N^{2n})$,
$i = 1,2$, with $[F] \in H^2 (N^{2n}, \Z)$.
If  for $n =2$
$$
\left (\sum_{i,j=1}^2 \gamma_{ij}  F_i \wedge F_j \right )  <0
$$
or  for  $n > 3$
$$
\left (\sum_{i,j=1}^2 \gamma_{ij}  F_i \wedge F_j \right )  \wedge  \Omega_{N^{2n }}^{n -2}  =0
$$
for some positive definite
matrix $(\gamma_{ij}) \in M_2 (\R)$, then by \cite[Proposition 4.5]{Sw} there is a compact simply
connected $\T^2$-bundle $\tilde W$ over $N^{2n}$ whose total space admits a $1$-st Gauduchon metric.
Note that the condition $$ \sum_{i,j=1}^2 \gamma_{ij}  F_i \wedge F_j  =0$$ is equivalent to the
condition that $1$-st Gauduchon metric on $\tilde W$  is SKT.

The manifold $\tilde W$ is the universal covering
of the twist $W$ of $N^{2n} \times \T^2$, where the K\"ahler flat
metric over $\T^2 = \C/\Z^2$ is given by the matrix $(\gamma_{ij})$
with respect to the standard generators with a compatible complex
structure and
topologically $W$ is a principal torus bundle over $N^{2n}$ with Chern classes $[F_i]$.

\begin{example} {\rm Consider as  in \cite[Example 4.6]{Sw}  a simply-connected projective K\"ahler manifold of real dimension $4$ and fix an imbedding of $N_0$ in  ${\mathbb {C P}}^r$.  Then a generic  linear subspace $L$ of  complex dimension $r -2$ in ${\mathbb {C P}}^r$
intersects transversely $N_0$ in a finite number of points $p_1, \dots, p_d$. 
One may choose homogeneous coordinates  $[z_0, \dots, z_r]$ on ${\mathbb {C P}}^r$ in such a way $L$ is defined by $z_0 = z_1 =0$. Let now $\tilde {{\mathbb {C P}}^r} \subset  {\mathbb {C P}}^r \times {\mathbb {C P}}^1$
be the complex blow-up of ${\mathbb {C P}}^r$ along $L$, and denote by $\pi_1$ and $\pi_2$ respectively the projections on the first and on the second factor of ${\mathbb {C P}}^r \times {\mathbb {C P}}^1$. As shown in \cite[Example 4.6]{Sw} if we define $N_1 = \pi^{-1} (N_0)$, then $f \vert_{N_1}$ defines a $(1,1)$-form $F_1$ on $N_1$ such that the cohomology class $[F_1]$  is non-zero and $F_1^2 =0$. If one iterates this construction one gets another  non-zero  cohomology class $[F_2]$  such that $F_2^2 =0$. Now, since $F_1 \wedge F_2 \neq 0$, one can choose a positive matrix  $(\gamma_{ij}) \in M_2 (\R)$ such  that $\sum_{i,j=1}^2 \gamma_{ij} F_i \wedge  F_j <0$. Then, by applying Proposition~\ref{curve} and Corollary~10 of \cite{FWW},
one gets a $1$-st Gauduchon metric on the universal covering
of the twist of $N_1 \times \T^2$.
}
\end{example}

\vskip1cm

\noindent {\bf Acknowledgments.}
We would like to thank S. Ivanov for useful comments and suggestions about the proof of Theorem~\ref{conform-balanced}.
This work has been partially supported through Project MICINN (Spain) MTM2008-06540-C02-02,
Project MIUR ``Riemannian Metrics and Differentiable Manifolds" and by GNSAGA of INdAM.

\smallskip

{\small

\end{document}